\documentclass[12pt]{amsart}
\usepackage{graphicx}
\usepackage{amssymb}
\usepackage{epstopdf}

\usepackage{amsmath}
\usepackage{amsthm}
\usepackage{a4wide}
\usepackage{pdfsync}
\usepackage{hyperref}
\usepackage{color}

\usepackage[pagewise,mathlines]{lineno}

\newtheorem{definition}{Definition}[section]
\newtheorem{proposition}{Proposition}[section]

\newtheorem{lemma}{Lemma}[section]
\newtheorem{theorem}{Theorem}[section]
\newtheorem{remark}{Remark}

\def\rr{\mathbb{R}}

\def\eps{\epsilon}

\def\mf{\mathcal{F}}
\def\ou{\overline u}

\def\la{\lambda}

\title[Nonlocal diffusion - convection]{Large time behaviour for a nonlocal diffusion - convection equation related with the gas dynamics }
\author[L. I. Ignat, A. Pazoto ]{Liviu I. Ignat and Ademir Pazoto}

\address{L. I. Ignat
\hfill\break\indent Institute of Mathematics ``Simion Stoilow'' of the Romanian Academy\\
\hfill\break\indent  21 Calea Grivitei Street \\010702 Bucharest \\ Romania
\hfill\break\indent \and
\hfill\break\indent
 BCAM - Basque Center for Applied Mathematics\\
 \hfill\break\indent  Alameda de Mazarredo, 14
E-48009 Bilbao, Basque Country - Spain.}

 \email{{\tt
liviu.ignat@gmail.com}\hfill\break\indent  {\it Web page: }{\tt
http://www.imar.ro/\~\,lignat}}

\address{A. Pazoto
\hfill\break\indent Instituto de Matem\'atica, Universidade Federal do Rio de Janeiro\\
\hfill\break\indent  P.O. Box 68530 \\ CEP 21945-970, Rio de Janeiro, RJ\\Brazil.}

 \email{{\tt
ademir@im.ufrj.br\hfill\break\indent  }}

\begin{document}

\begin{abstract}
In this paper we consider a model that involves nonlocal diffusion and a classical convective term. Using  a scaling argument and a new compactness argument we  obtain the first term in the asymptotic behavior of the solutions.
\end{abstract}

\maketitle

\section{Introduction}

Let us consider the following system
\begin{equation}\label{gases}
\left\{
\begin{array}{ll}
u_t+uu_x=-q_x,& x\in \rr, t>0,\\
-q_{xx}+q=-u_x, &x\in \rr,
\end{array}
\right.
\end{equation}
with $u_0$ as initial condition. This model has been considered by Hamer \cite{HAMER1971} in the context of radiating gases.
System \eqref{gases} can be rewritten as a nonhomogeneous Burgers equation of the form
\begin{equation}\label{burgers}
u_t+(\frac{u^2}{2})_x= K\ast u-u,
\end{equation}
where  kernel $K$  is given by $\frac{1}{2}e^{-|x|}$. Regarding  well-posedness, it has been proved that for initial data in $L^1(\rr)\cap L^\infty(\rr)$ there is a unique entropy solution for \eqref{burgers} (see, for instance, \cite{1052.35126}, \cite{0793.76005}, \cite{MR2103702}).
In \cite{1070.35065} the asymptotic behavior of the solutions is studied and it is proved that the solution behaves as the solution of the viscous Bourgers equation with Dirac delta initial data.

The aim of this paper is to extend the previous result to the following system
\begin{equation}\label{gas}
\left\{
\begin{array}{ll}
u_t=J\ast u-u-a (|u|^{q-1}u)_x,& x\in \rr,\ t>0,\\[10pt]
u(0)=\varphi.
\end{array}
\right.
\end{equation}
where $q\geq 2$, $a$ is a real number and $J\in L^1(1+|x|^2)$ is a continuous nonnegative, radially symmetric function with mass one. We  consider here the case $a=1$ and
positive solutions since the general case can be treated in a similar way. The multidimensional case can be treated in a similar way.

The main result of this paper is the following one.
\begin{theorem}\label{asimp}Let $1\leq p<\infty$.
For any $\varphi\in L^1(\rr)\cap L^\infty(\rr)$ the solution $u$ of system \eqref{gas} satisfies
\begin{equation}\label{lim.t}
\lim _{t\rightarrow \infty} t^{-\frac  12(1-\frac 1p)}\|u(t)-u_M(t)\|_{L^p(\rr)}=0
\end{equation}
where $u_M$ is the unique solution of the heat (when $q>2$) equation
\begin{equation}
\left\{
\begin{array}{ll}
w_t - A  w_{xx} = 0,& x\in \rr, t>0,\\[10pt]
w(x,0)=M\delta_0, &x\in \rr,
\nonumber
\end{array}
\right.
\end{equation}
 respectively, viscous Burgers equation (when $q=2$)
 \begin{equation}
\left\{
\begin{array}{ll}
w_t - A  w_{xx} + ({w^2})_x= 0,& x\in \rr, t>0,\\[10pt]
w(x,0)=M\delta_0, &x\in \rr,\nonumber
\end{array}
\right.
\end{equation}
  where $M$ is the mass of the initial data and
$$A=\frac 12\int_{\rr}z^2J(z)dz.$$
\end{theorem}

We point out that when $q=2$ the asymptotic behavior is obtained in \cite{1070.35065} by  crucially using  the Oleinik  estimate $u_x\leq 1/t$, which is not  available in the case $q>2$ and new techniques have to be developed. Even if the case $q=2$ can be treated as in \cite{1070.35065} we prefer to give a proof that covers both the critical and the super-critical case $q\geq 2$.

We obtain the asymptotic behavior in Theorem \ref{asimp} by using a scaling method, i.e. we introduce a family of scaled solutions $\{u_\lambda\}_{\lambda>0}$ and reduce the asymptotic property \eqref{lim.t} to the strong convergence of the scaled family  (see for example \cite{MR2739418}). However, there is an important difficulty, namely we have to obtain the compactness of the family $\{u_\lambda\}_{\lambda>0}$. The difficulty comes from the lack of the smoothing effect of the linear semigroup
involved here. This has been already observed in the case of nonlocal evolution problems \cite[Ch. 1]{1214.45002}. In view of these difficulties, to obtain the compactness of the trajectories $\{u_\lambda(t)\}_{\lambda>0}$ we need to obtain a compactness argument that does not involve derivatives of $u_\lambda$ but  some nonlocal quadratic forms instead. The main tool to obtain the desired compactness is given in Section \ref{proofs} in Proposition \ref{main}.

The lack of smoothing effect of the linear semigroup does not allow us to apply the techniques used in \cite{0762.35011} for the convection-diffusion equation $u_t=u_{xx}+a(|u|^{q-1}u)_x$. In \cite{0762.35011} the authors use the fact that the heat kernel is smooth and prove that  the nonlinearity decays faster than the linear part and then the asymptotic behavior is given by the linear semigroup.

Our results presented here could be compared with the ones in \cite{1052.35126} where the hyperbolic-parabolic and hyperbolic-hyperbolic relaxation limits are studied. In the cited paper the authors do not rescale the initial datum. As a consequence, for $q=2$,  they obtain the limit profile to be the solution of the viscous Burgers equation with initial data $\varphi$ instead of $M\delta_0$.


The extension of the results presented here  to the case $q\in (1,2)$ remains an open problem. In the classical convection-diffusion problem the case $q\in (1,2)$ has been considered in
 \cite{MR1233647}.

 The paper is organized as follows: in Section \ref{prelim} we present some preliminary results on system \eqref{gas} and present some known compactness results that will be useful in the proof of Theorem \ref{asimp}. In Section \ref{proofs} we give the proof of the main result of this paper, Theorem \ref{asimp}.

\section{prelimiaries}\label{prelim}

For the sake of completeness we present some results obtained in \cite{1052.35126} regarding the well-posedness of system \eqref{gas}. We pointed out that the authors proved the existence and uniqueness of a weak entropy solution and, as usual, such a class of solution is defined as follows:

\begin{definition} Let $f(u)$ be a $C^2$ flux and $J$ a nonnegative, even function with mass one.
A bounded measurable function $u$ is a weak solution to
\begin{equation}\label{claws}
\left\{
\begin{array}{ll}
u_t+(f(u))_x=J\ast u-u,& x\in \rr,t>0,\\[10pt]
u(0)=\varphi.
\end{array}
\right.
\end{equation}
 if it verifies this relation in the distributional sense and the test functions are smooth with compact support, intersecting the line $t=0$. A weak solution is entropic if, in addition, it verifies the inequality
\begin{equation}\label{def}
\begin{array}{c}
\displaystyle\int_0^T\int_{-\infty}^{+\infty}[\eta(u)\psi_t + q(u)\psi_x]dx dt + \int_0^T\int_{-\infty}^{+\infty}\eta(u_0(x))\psi(x,0)dx\\
\geq \displaystyle\int_0^T\int_{-\infty}^{+\infty}\eta'(u)[u-J\ast u]\psi dx dt,
\end{array}
\end{equation}
for any convex entropy $\eta$ with flux $q$ given by
$q'(s)=f'(s)\eta'(s)$
and for any nonnegative Lipschitz continuous test function $\psi$ on $\rr\times [0,T]$ with compact support, intersecting the line $t=0$.
\end{definition}

\begin{remark}\label{latt1} From the classical theory of scalar conservation laws, it is known that it is sufficient to obtain the above inequality for $\eta(u)=\pm u$, together with the family $\eta_k(u)=(u-k)^+$, $k\in \rr$, where $w^+=\max\{w,0\}$.
\end{remark}

Using the vanishing viscosity method in \cite{1052.35126} the authors  prove the following result.

\begin{theorem}\label{ex-latt}For any $u_0\in L^1(\rr^d)\cap L^\infty(\rr^d)$ and $T>0$, there exists a unique function $u\in C([0,T], L^1(\rr))\cap L^\infty( [0,T]\times \rr)$ entropy solution of \eqref{claws}. Moreover, the solution is global. For any nonnegative initial data the solution is nonnegative.
\end{theorem}

%
%
%

%
%

We  now recall some results about the characterization of compact sets in $L^p(0,T,B)$ where $B$ is a Banach space.
\begin{theorem}[\cite{MR916688}, Th.~1, p.~71] \label{simon}
Let $\mf \subset L^p(0,T,B)$. $\mf$ is relatively compact in $L^p(0,T,B)$ for $1\leq p<\infty$, or $C(0,T,B)$ for $p=\infty$ if and only if
\begin{enumerate}
\item \label{condition.1}
$\{\int _{t_1}^{t_2} f(t)dt, \, f\in \mf\}$ is  relatively compact in $B$ for all $0<t_1<t_2<T$,\\

\item  \label{condition.2}  $\|\tau _h f-f\|_{L^p(0,T-h,B)}\rightarrow 0$ as $h\rightarrow 0$ uniformly for $f\in \mf$.
\end{enumerate}
\end{theorem}

Also the following compactness criterion  is given.

\begin{theorem}[\cite{MR916688}, Th.~5, p.~84]\label{3spaces}
Let us consider three Banach spaces  $X\hookrightarrow B \hookrightarrow Y$ where the embedding $X\hookrightarrow B$ is compact.
Assume $1\leq p\leq \infty$ and \\
i) $\mf$ is bounded in $L^p(0,T,X)$\\
ii) $\|\tau _h f-f\|_{L^p(0,T-h,Y)}\rightarrow 0$ as $h\rightarrow 0$ uniformly for $f\in \mf$.

Then $\mf $ is relatively compact in $L^p(0,T,B)$ (and in $C(0,T,B)$ is $p=\infty$).
\end{theorem}

This last criterion is obtained  by using Theorem \ref{simon} and the following inequality that follows from the fact that space $X$ is compactly embedded in space $B$: for any $\eps>0$ there exists $\eta(\eps)>0$ such that
\begin{equation}\label{ineg.eps}
\|u\|_B\leq \eps \|u\|_X+\eta(\eps) \|u\|_Y,\quad \forall \, u \in X.
\end{equation}
In the nonlocal setting we will give a similar result in Lemma \ref{balance} in Section \ref{proofs}.

\medskip
We recall now some compactness results that have been proven in the nonlocal context \cite{1103.46310}, \cite{1214.45002}.

\begin{theorem}[\cite{1214.45002}, Th.~6.11, p.~126]\label{rossi}
Let $1\leq p<\infty$ and $\Omega\subset \rr$ open. Let $\rho:\rr\rightarrow \rr$ be a nonnegative smooth continuous radial function with compact support, non identically zero, and $\rho_n(x)=n\rho(nx)$. Let $\{f_n\}_{n\geq 1}$ be a  bounded sequence in $L^p(\rr)$ such that for all $n\geq 1$
\begin{equation}\label{ros.1}
\int _{\Omega}\int _{\Omega} \rho_n(x-y) |f_n(x)-f_n(y)|^pdxdy\leq \frac {M}{n^p}.
\end{equation}
The following hold:

1. If $\{f_n\}_{n\geq 1}$ is weakly convergent in $L^p(\Omega)$ to $f$ then $f\in W^{1,p}(\Omega)$ for $p>1$ and $f\in BV(\Omega)$ for $p=1$.

2. Assuming  that $\Omega$ is a smooth bounded domain in $\rr$ and $\rho(x)\geq \rho(y)$ if $|x|\leq |y|$ then $\{f_n\}_{n\geq 1}$ is relatively compact in $L^p(\Omega)$.
\end{theorem}

\begin{remark}
We point out that the compactness assumption on the function $\rho$ could be removed. In fact once we have estimate \eqref{ros.1} for $\rho$ we also have it for any other compactly supported function $\tilde \rho$ with $\tilde \rho\leq \rho$.
\end{remark}

%
%
%

Using the same arguments as in the proof of Theorem \ref{rossi} (see \cite[Ch.~6, p.~128]{1214.45002}) we obtain the following result.
\begin{theorem}\label{rossi-time}
Let $1\leq p<\infty$ and $\Omega\subset \rr$ open. Let $\rho:\rr\rightarrow \rr$ be a nonnegative smooth continuous radial functions with compact support, non identically zero, and $\rho_n(x)=n\rho(nx)$. Let $\{f_n\}_{n\geq 1}$ be a bounded sequence of functions in $L^p((0,T)\times \Omega)$ such that
\begin{equation}\label{ros.1.time}
\int _0^T\int _{\Omega}\int _{\Omega} \rho_n(x-y) |f_n(x)-f_n(y)|^pdxdy\leq \frac {M}{n^p}.
\end{equation}
If  $\{f_n\}_{n\geq 1}$ is weakly convergent in $L^p((0,T)\times \Omega)$ to $f$ then $f\in L^p((0,T),W^{1,p}(\Omega))$.
%
%
%
\end{theorem}

The proof of this theorem follows the lines of Theorem \ref{rossi} in \cite{1214.45002} and we will omit it here.
The result in Theorem \ref{rossi-time} will be used later in Section \ref{proofs} in the proof of the main result of this paper  - Theorem \ref{asimp}.

\section{Proof of the main result}\label{proofs}

We first state some results regarding the behavior of the $L^p(\rr)$-norms of the solutions of system \eqref{gas}. We refer to \cite{Ignat:2013fk, MR2356418} for similar results in the case of the equation
$$u_t=J\ast u-u+G\ast |u|^{q-1}u-|u|^{q-1}u, \quad q\geq 2.$$

\begin{theorem}\label{decay.simple}
For any $\varphi\in L^1(\rr)\cap L^\infty(\rr)$ and $1\leq p<\infty$ the solution of system \eqref{gas} satisfies:
\begin{equation}\label{decay.u.1}
\|u(t)\|_{L^p(\rr)}\leq C(p,J,\|\varphi\|_{L^1(\rr)}, \|\varphi\|_{L^\infty(\rr)})(1+t)^{-\frac 12(1-\frac 1p)}, \ t\geq 0.
\end{equation}
\end{theorem}
\begin{proof}
The main idea of the proof  is to adapt the so-called Fourier splitting methods used previously by Schonbek \cite{MR1356749}  (see also \cite{MR2103702}, p. 26).
Multiplying the equation \eqref{gas} by $u$ and integrating by parts we get
\[
\frac {d}{dt}\int _{\rr}u^2(t,x)dx+\frac 12\int _{\rr}\int _{\rr}J(x-y)(u(x)-u(y))^2dxdx=0.
\]
Remark that the assumptions on $J$ guarantee that its Fourier transform satisfies
\[
\left\{
\begin{array}{ll}
\hat J(\xi)\leq 1-\frac {\xi^2}2,& |\xi |\leq R,\\[10pt]
\hat J(\xi)\leq 1-\delta, & |\xi|\geq R.
\end{array}
\right.
\]
  The first inequality above is a consequence of the fact the $J$ has a second momentum. The second one is a consequence of the fact that
$J$ is nonnegative and has mass one.
Proceeding as in \cite{MR2356418} we find that
\begin{equation}\label{decay.u.2}
\|u(t)\|_{L^2(\rr)}\lesssim \frac{\|\varphi\|_{L^2(\rr)}}{(t+1)^{1/2}}+\frac{\|\varphi\|_{L^1(\rr)}}{(t+1)^{1/4}},
\end{equation}
which proves estimate \eqref{decay.simple} for $p=2$.

The general case follows by an inductive argument and using that
\begin{align*}
\frac {d}{dt}\int _{\rr}u^p(t,x)dx&=-p \int _{\rr}(J\ast u-u)u(t,x)^{p-1}dx\\
&\leq -\frac{2(p-1)}{p}\int _{\rr}\int _{\rr} J(x-y)(u(t,x)^{p/2}-u(t,y)^{p/2})^2dxdy.
\end{align*}
The complete details are given in \cite{MR2356418}.
\end{proof}

\begin{proof}[Proof of Theorem \ref{asimp}]
In order to obtain the long time behavior of the solutions of system \eqref{gas} we use the method of self-similar solutions.
Let us introduce the scaled functions
$$u_\lambda(t,x)=\lambda u(\lambda^2t,\lambda x),\, \varphi_\lambda(x)=\lambda \varphi(\lambda x)  \quad \text{and} \quad J_\la(x)=\la J(\la x).$$
 Then $u_\la$ satisfies the system
\begin{equation}\label{gas-scalat}
\left\{
\begin{array}{ll}
u_{\la,t}=\la^2 (J_\la\ast u_\la-u_\la)-\la ^{2-q}(u_\la^q)_x,& x\in \rr,t>0,\\[10pt]
u_\la(0,x)=\varphi_\la(x),& x\in \rr.
\end{array}
\right.
\end{equation}
We remark that for $p=1$, limit \eqref{lim.t} is equivalent with the fact that
\begin{equation}\label{limit.u.lambda}
\lim _{\lambda\rightarrow \infty}\|u_\lambda(1)-u(1)\|_{L^1(\rr)}=0.
\end{equation}
In the following we will prove \eqref{limit.u.lambda}. The case $p\in (1,\infty)$ will be obtained in
Step IV of our proof.

%

%
%


\medskip

\textbf{Step 1. Compactness of $\{u_\lambda\}_{\la>0}$ in $L^1_{loc}((0,\infty)\times \rr)$.}
Using Theorem \ref{decay.simple} we have that for any $1\leq p<\infty$ there exists a positive constant $C=C(p, J, \|\varphi\|_{L^1(\rr)}, \|\varphi\|_{L^\infty(\rr)})$
such that $u_\lambda$ the solution of system \eqref{gas} satisfies
\begin{equation}\label{decay.u.lam}
\|u_\la(t)\|_{L^p(\rr)}\leq \frac{C\lambda ^{1-\frac 1p}}{(1+t\lambda^2)^{\frac 12(1-\frac 1p)}}\leq C t^{-\frac 12(1-\frac 1p)}.
\end{equation}
Hence  for any $0<t_1<t_2$, the family $\{u_\la\}_{\lambda>0}$ is uniformly bounded in $L^\infty((t_1,t_2),L^2(\rr))$.
Moreover if we multiply equation \eqref{gas-scalat} by $u_\la$ and integrate on the space variable $x$ we obtain that
\begin{align*}
\frac 12\frac {d}{dt}\int _{\rr}u_\la^2(t,x)dx&=\la^2 \int _{\rr} (J_\la\ast u_\la(t)-u_\la(t))u_\la(t)dx\\
&=-\frac {\lambda^2}2
\int _{\rr}\int _{\rr} J_\la (x-y)(u_\la(t,x)-u_\la(t,y))^2dxdy.
\end{align*}
Hence for any $0<t_1<t_2$  the following identity holds
\begin{equation}\label{energy.est}
\int _{\rr}u_\la^2(t_2,x)dx+ {\lambda^2} \int _{t_1}^{t_2}\int _{\rr}\int _{\rr} J_\la (x-y)(u_\la(t,x)-u_\la(t,y))^2dxdydt=
\int _{\rr}u_\la^2(t_1,x)dx.
\end{equation}
We now estimate the $H^{-1}(\rr)$-norm of $u_{\lambda,t}$. We need the following Lemma.

\begin{lemma}\label{est.ariba}For any $\rho\in L^1(\rr,|x|^2)$
the following inequality
\begin{equation}\label{ineg-1}
\lambda ^3\iint _{\rr^{2}} \rho(\lambda (x-y)(\varphi(x)-\varphi(y))^2dxdy\leq \int _{\rr} \rho(x)|x|^2 dx\int _{\rr}  \varphi_x^2dx
\end{equation}
holds for all $\lambda >0$ and  $\varphi\in H^1(\rr)$.
\end{lemma}

\begin{proof}[Proof of Lemma \ref{est.ariba}]
Making changes of variables it is sufficient to consider the case $\la=1$. Since
$$\varphi(x)-\varphi(y)=\int _0^1 (x-y)\cdot  \varphi' (y+s(x-y))ds$$
we get that
\begin{align*}
\iint _{\rr^{2}} \rho(x-y)(\varphi(x)-\varphi(y))^2dxdy&\leq \iint _{\rr^{2d}} \rho(x-y)|x-y|^2  \int _0^1 | \varphi' (y+s(x-y))|^2dsdxdy\\
&= \int _{\rr}\rho(z)z^2dz \int _{\rr} \varphi_x^2dx,
\end{align*}
which proves the desired inequality.
\end{proof}

Denoting by $\langle , \rangle _{-1,1}$ the duality product we obtain that
the time derivative of  $u_{\la}$ verifies
\begin{align*}
 \langle u_{\la,t} &,\varphi\rangle  _{-1,1}=  \langle \la^2 (J_\la\ast u_\la -u_\la)-(u_\la^q)_x,\varphi\rangle  _{-1,1}\\
 &=-\frac {\la^2}2 \int _{\rr} \int _{\rr}  J_\la(x-y)(u_\la(x)-u_\la(y))(\varphi(x)-\varphi(y))dxdy-\int _{\rr} u_\la^q\varphi_x dx\\
 &\lesssim \Big(\la^2 \int _{\rr} \int _{\rr}  J_\la(x-y)(u_\la(x)-u_\la(y))^2\Big)^{1/2}
 \Big(\la^2 \int _{\rr} \int _{\rr}  J_\la(x-y)(\varphi(x)-\varphi(y))^2\Big)^{1/2}\\
 &\quad + \|u_\la ^q\|_{L^2(\rr)}\|\varphi\|_{H^1(\rr)}.
\end{align*}
Using  estimate \eqref{decay.u.lam}  and identity \eqref{energy.est} we obtain that  $\{u_{\la,t}\}_{\lambda>0}$ is  uniformly bounded in $L^2((t_1,t_2),H^{-1}(\rr))$.

The above estimates prove that there exists $M=M(t_1,\|\varphi\|_{L^1(\rr)},\|\varphi\|_{L^{\infty}(\rr)})$ such that
\begin{equation}\label{cond.1}
\|u_\la\|_{L^\infty([t_1,t_2],\, L^2(\rr))}\leq M,
\end{equation}
\begin{equation}\label{cond.2}
\la^2\int _{t_1}^{t_2}\int _{\rr}\int _{\rr} J_\la (x-y)(u_\la(t,x)-u_\la(t,y))^2dxdydt\leq M
\end{equation}
and
\begin{equation}\label{cond.3}
\|u_{\la,t}\|_{L^2([t_1,t_2],H^{-1}(\rr))}\leq M.
\end{equation}

The above estimates bring back to our mind the  Aubins-Lions compactness criterion given in Theorem \ref{3spaces} in Section \ref{prelim} applied on the spaces $H^1, L^2$ and $H^{-1}$.
However, instead of the $H^1$-norm, in \eqref{cond.2} we have an integral estimate that seems very similar to the ones used in \cite{1103.46310} to characterize functions in Sobolev spaces.
We will adapt  the proof of Theorem \ref{3spaces} to our nonlocal setting and provide a nonlocal compactness criterion in the following proposition. The proof of Proposition \ref{main} will be given later.

\begin{proposition}\label{main}Let $\rho$ be a function as in Theorem \ref{rossi-time} and $\rho_n(x)=n\rho(nx)$. Assume
sequence $\{f_n\}_{n\geq 1}$ has the following properties
\begin{equation}\label{hyp.1}
\|f_n\|_{L^\infty((0,T),L^2(\rr))}\leq M,
\end{equation}
\begin{equation}\label{hyp.2}
n^2\int _0^T \int_{\rr}\int _{\rr} \rho_n(x-y)(f_n(t,x)-f_n(t,y))^2dxdydt\leq M
\end{equation}
and
\begin{equation}\label{hyp.3}
\|\partial_tf_n\|_{L^2((0,T),H^{-1}(\rr))}\leq M.
\end{equation}
Then there exists a function $f\in L^2((0,T),H^1(\rr))$ such that, up to a subsequence,
\begin{equation}\label{conclusion}
f_n\rightarrow f\quad \text{in}\quad  L^2_{loc}((0,T)\times \rr).
\end{equation}
\end{proposition}

\begin{remark}
The main difference between the results in Proposition \ref{main} and the ones in Theorem \ref{3spaces} is that  instead of having a fix space $X=H^1(\Omega)$ that is compactly embedded in $B=L^2(\Omega)$ we have in \eqref{hyp.2} a family of nonlocal quadratic forms. A more general result involving the $L^p$-norms has recently been obtained  in \cite{Ignat:2013fk}.
\end{remark}

Using that  $J$ is a nonnegative  continuous function with mass one we can choose a function $\rho$ as in Theorem \ref{rossi-time}
such that $\rho\leq J$. Thus we can
 apply Proposition \ref{main} to the time interval $[t_1,t_2]$ and to the family $\{u_\lambda\}_{\la>0}$.
It follows that there exists a function $\ou\in L^2([t_1,t_2],H^1(\rr))$ such that
 $u_\la\rightarrow \ou$  in $ L^2([t_1,t_2];L^2_{loc}( \rr))$.
 This shows that
 $$u_\la\rightarrow \ou \ \text{in} \ L^1_{loc}((0,\infty)\times  \rr))$$
and, therefore,
$$u_\lambda \rightarrow u \mbox{ a.e. in } (0,\infty) \times \rr.$$
This is not sufficient to pass to the limit in the equation, since the last convergence was not obtained for all $t>0$.
To overcome this difficulty  we will prove that for any $t>0$, sequence $u_\lambda(t)$ weekly converges to $\ou(t)$ in $L^2(\rr)$.
%
The estimates obtained above
yield for $\{u_{\lambda,t}\}$ a bound in $L^2_{loc}((0,\infty),H^{-1}_{loc}(\rr))$. As a consequence
of \eqref{cond.1} we also have that $\{u_\lambda\}$ is bounded in $L^\infty_{loc}((0,\infty), L^2_{loc}(\rr)$.
Taking into account that $L^2_{loc}(\rr)$ is compactly embedded in $H^{-\varepsilon}_{loc}(\rr)$ for any $\varepsilon > 0$
and that for any $0<\eps<1$ the space
  $H^{-\varepsilon}_{loc}(\rr)$ is continuously embedded in $H^{-1}_{loc}(\rr)$, we deduce (see \cite[Corollary.~4, p.~85]{MR916688}) that
$\{u_\lambda\}_{\la>0}$ is relatively compact in $C_{loc}((0,\infty),H^{-\varepsilon}_{loc}(\rr))$.
Consequently, we can extract a subsequence, not relabeled, in such a way that
\begin{equation}
\begin{array}{l}\label{id.lim-4}
u_\lambda \rightarrow \ou\mbox{ in } C_{loc}((0,\infty),H^{-\varepsilon}_{loc}(\rr)).
\end{array}
\end{equation}
On the other hand, for any $t>0$
estimate \eqref{cond.1} guarantees the existence of a function $v(t)\in L^2(\mathbb{R})$, such that up to a subsequence
\begin{equation}\label{id.lim-5}
u_\lambda(t)\rightharpoonup v(t)\mbox{ weakly in } L^2(\mathbb{\rr}).
\end{equation}
The uniqueness of the limit in $\mathcal{D}'(\mathbb{R})$ guarantees that
\begin{equation}\label{weak}
u_\lambda(t)\rightharpoonup \ou(t)\mbox{ weakly in } L^2(\mathbb{R}), \mbox{ for all } t > 0.
\end{equation}

Moreover, we recall that in view of \eqref{decay.u.lam} there exists a function $\chi$ such that for any $1<p<\infty$
$$u_\lambda^q \rightharpoonup \chi \quad \text{in} \quad L^p_{loc}((0,\infty),L^p(\rr)).$$
Since $u_\lambda \rightarrow \ou$ a.e. in $(0,\infty)\times \rr$ we can conclude that the nonlinear term satisfies
\begin{equation}\label{conv.non}
u_\lambda^q\rightharpoonup \ou^q\quad \mbox{  in }\quad  L^p_{loc}((0,\infty)\times \mathbb{R}).
\end{equation}

\textbf{Step 2. Compactness of $\{u_\lambda\}$ in $L^1_{loc}( (0,\infty),L^1(\rr))$.}
Based on the previous step we conclude that for any positive  $\tau$ and $R$  the family
$\{u_\lambda\}_{\la>0}$  is relatively compact in $L^1( (\tau ,T) \times (-R,R))$.
Using a standard diagonal argument the compactness in $L^1( (\tau ,T) \times \rr)$  is reduced to the uniform  control of the tails
of the family $\{u_\lambda\}_{\lambda>1}$:
\begin{equation}\label{est.fuera}
\int _{\tau}^T \|u_\lambda (t)\|_{L^1(|x|>R)}dt\rightarrow 0 \quad \text{as}\quad R\rightarrow \infty, \, \text{uniformly in }\, \lambda\geq1.
\end{equation}
This follows from the following Lemma.

\begin{lemma}\label{est.2r}
There exists a constant $C=C(J,\|\varphi\|_{L^1(\rr)},\|\varphi\|_{L^\infty(\rr)})$ such that the following
\begin{equation}\label{int.2r}
\int _{|x|>2R} u_\lambda (t,x)dx\leq \int _{|x|>R}\varphi(x)dx+C(\frac{t}{R^2}+\frac {t^{1/2}}R)
\end{equation}
 holds for any $t>0$, $R>0$, uniformly on $\lambda\geq 1$.
\end{lemma}

\begin{proof}
Let $\psi\in C^\infty(\rr)$ be a nonnegative function that satisfies $\psi(x)\equiv 0$ for $|x|<1$ and $\psi(x)\equiv 1$ for $|x|>1$. We put $\psi_R(x)=\psi(x/R)$.
We multiply equation \eqref{gas-scalat} by $\psi_R$ and integrate by parts to obtain
\begin{align*}
\int _\rr u_\lambda (t,x)\psi_R(x)dx-\int_{\rr}\varphi_{\lambda} \psi_R(x)dx=&
\lambda^2\int _0^t \int_{\rr} u_\lambda (s,x)  (J_\lambda\ast \psi_R -\psi_R)dxds\\
&+ \lambda^{2-q}\int _0^t \int _{\rr} u_\lambda^q (s,x)(\psi_R)_x(x)dxds.
\end{align*}
We now use that
$$\|\lambda ^2(J_\lambda\ast \psi_R-\psi_R)\|_{L^\infty(\rr)}\leq C(J)\|(\psi_R)_{xx}\|_{L^\infty(\rr)}=C(J)R^{-2}\|\psi_{xx}\|_{L^\infty(\rr)},$$
$$\|(\psi_R)_{x}\|_{L^\infty(\rr)}= R^{-1}\|\psi_{x}\|_{L^\infty(\rr)}$$
and the conservation of the $L^1(\rr)$-norm of $u_\lambda$
to find that
\begin{align*}
\int _\rr u_\lambda (t,x)\psi_R(x)dx &\leq \int_{\rr}\varphi_{\lambda} \psi_R(x)dx+
C(J)R^{-2}\|\psi_{xx}\|_{L^\infty(\rr)}\int _0^t \int_{\rr} u_\lambda (s,x) dsdx\\
&\quad + \lambda^{2-q}R^{-1}\|\psi_{x}\|_{L^\infty(\rr)}\int _0^t \int _{\rr} u_\lambda^q (s,x)dxds.\\
&\leq \int_{|x|>R}\varphi_{\lambda}(x)dx+
C(J)R^{-2}\|\psi_{xx}\|_{L^\infty(\rr)}t \|\varphi\|_{L^1(\rr)}\\
&\quad + \lambda^{2-q}R^{-1}\|\psi_{x}\|_{L^\infty(\rr)}\int _0^t \int _{\rr} u_\lambda^q (s,x)dxds.
\end{align*}
To estimate the last term in the above inequality we use the decay of the solution $u_\lambda$ as given by \eqref{decay.u.lam} and obtain that
$$\lambda^{2-q}\int _0^t \int _{\rr} u_\lambda^q (s,x)dxds\lesssim \lambda \int _0^t\frac {ds}{(1+\lambda^2 s)^{\frac {q-1}2}}
=\lambda^{-1}\int _0^{t\lambda ^2} \frac{ds}{(1+ s)^{\frac {q-1}2}}.$$
Since for any $q\geq 2$
$$\lim _{x\rightarrow 0} x^{-1}\int _0^{x^2}\frac{ds}{(1+ s)^{\frac {q-1}2}}=0, \quad \text{and}\quad
\lim _{x\rightarrow \infty} x^{-1}\int _0^{x^2}\frac{ds}{(1+ s)^{\frac {q-1}2}}= \lim _{x\rightarrow \infty} \frac{2x}{(1+ x^2)^{\frac {q-1}2}}<\infty$$
we find that
$$\lambda^{2-q}\int _0^t \int _{\rr} u_\lambda^q (s,x)dxds\lesssim Ct^{1/2}.$$
Since $\lambda>1$ we get
\begin{equation}\label{int.2r-1}
\int _\rr u_\lambda (t,x)\psi_R(x)dx\leq \int _{|x|>\lambda R} \varphi(x)dx+ C(\frac{t}{R^2}+\frac {t^{1/2}}R)
\leq \int _{|x|> R} \varphi(x)dx+ C(\frac{t}{R^2}+\frac {t^{1/2}}R)
\end{equation}
and the proof of the Lemma is finished.
\end{proof}

\textbf{Step 3. Identification of the limit.}
Our aim here is to pass to the weak limit in the equation involving $u_\lambda$, as well as, to identify the corresponding initial condition.
Let us choose $0<\tau<t$. We  multiply  equation \eqref{gas-scalat}  by $\phi \in {C}^\infty_c(\mathbb{R})$ and integrate over $(0,\tau)\times \mathbb{R}$ to obtain
\begin{align}\label{id.lim}
\int _\rr u_\lambda (t,x)\phi(x)dx-\int_{\rr}u_{\lambda}(\tau,x)\phi(x)dx=&
\lambda^2\int _\tau^t \int_{\rr} (J_\lambda\ast \phi -\phi)(x)u_\lambda (s,x)  dxds\\
&+ \lambda^{2-q}\int _\tau^t \int _{\rr} u_\lambda^q (s,x)\phi_x(x)dxds.\nonumber
\end{align}
The convergence of the terms on the left hand side of \eqref{id.lim} follows from the weak convergence  obtained in \eqref{weak}. In order to pass to the limit the terms on the right hand side, we  use the dominated convergence theorem and that $u_\lambda\rightarrow \ou$ in $L^1_{loc}((0,\infty),L^1(\rr))$
to obtain
\begin{align}\label{id.lim-1}
\lambda^2\int _\tau^t \int_{\rr} (J_\lambda\ast \phi -\phi)(x)u_\lambda (s,x)dxds  \rightarrow
A\int _\tau^t \int_{\rr} \ou \phi_{xx} dxds,\mbox{ as } \lambda\rightarrow\infty,
\end{align}
where
$$A =\frac 12 \int_{\rr} J(z)z^2 dz.$$

Let us consider for the moment the case $q>2$.
Using  \eqref{conv.non}  and the fact that  $q>2$
we have
\begin{align}\label{id.lim-2}
\lambda^{2-q}\int _\tau^t \int _{\rr} u_\lambda^q (s,x)\phi_x(x)dxds\rightarrow 0,\mbox{ as } \lambda\rightarrow\infty.
\end{align}
Returning to \eqref{id.lim}  we conclude that for all  $\phi \in {C}^\infty_c(\mathbb{R})$ the limit function $\ou$ satisfies
\begin{align}\label{id.lim-3.2}
\int _\rr \ou(t,x)\phi(x)dx-\int_{\rr}\ou(\tau,x)\phi(x)dx=
A\int _\tau^t \int _{\rr}\ou(s,x) \phi_{xx}(x) dxds,\end{align}
i.e.
$$\ou_t=\ou _{xx} \quad \text{in}\quad  \mathcal{D'}((0,\infty)\times \rr).$$

Now it remains to identify in the initial condition in the equation satisfied by $\ou$. Following the computation which leads to \eqref{int.2r} with
$\psi_R$ replaced by $\phi\in {C}^\infty_c(\mathbb{R})$, we obtain
\begin{align}\label{id.lim-3}
\left|\int _\rr u_\lambda(t,x)\phi(x)dx-\int_{\rr}u_{0,\lambda}(x)\phi(x)dx\right|\leq C\,(t + t^{\frac{1}{2}}),\mbox{ for all }\phi \in {C}^\infty_c(\mathbb{R}),
\end{align}
where $C>0$ is independent of $\lambda$.
On the other hand, by change of variables and the dominated convergence theorem, we obtain that
$$\int_{\rr}\varphi_\lambda(x)\phi(x)dx = \int_{\rr}\varphi(x)\phi\left(\frac{x}{\lambda}\right)dx\rightarrow M\phi(0),\mbox{ as } \lambda\rightarrow\infty.$$
Hence, from \eqref{id.lim-3} it follows that
$$\lim _{t\rightarrow 0} \int _{\rr}\ou(t,x)\phi(x)dx =M\phi(0).$$
By a density argument the above limit also holds for any bounded continuous function $\phi$. This means that
$$\ou(x,0)=M\delta_0$$ in the sense of bounded measures.

Let us now prove that $u\in C((0,\infty), L^1(\rr)).$ Following the same steps as in Lemma \ref{est.2r} and letting  $\lambda\rightarrow \infty$ we obtain that
for any $0<s<t$ the following holds
\begin{align*}
\left|\int _\rr u(t,x)\phi(x)dx-\int_{\rr}u(s,x)\phi(x)dx\right|&\lesssim  |t-s|||\varphi||_{W^{2,\infty}(\mathbb{R})}.
\end{align*}
Let us recall that (see \cite{MR2759829}, p.~126)
$$||f||_{L^1(\mathbb{R})}=\sup\left\{\int_{\rr}f\varphi dx : \varphi \in {C}^\infty_c(\mathbb{R}),\,
||\varphi||_{L^\infty(\mathbb{R})}\leq 1 \right\}.$$
Now let us choose $\varepsilon > 0$. We know that there exists $\varphi_\eps \in  {C}^\infty_c(\mathbb{R})$ satisfying
\begin{align*}
||u(t) - u(s)||_{L^1(\mathbb{R})} \leq \varepsilon + \int_{\rr}(u(t) - u(s))\varphi_\varepsilon dx \leq \varepsilon
+|t-s|||\varphi_\varepsilon||_{W^{2,\infty}(\mathbb{R})},
\end{align*}
  Choosing $t$ close enough to point $s$ we obtain the continuity of $u$  at point $s$.

  Putting together the above results we obtain that  $\ou\in C((0,\infty),L^1(\rr))$ is the unique solution $u_M$ of the initial value problem
\begin{equation}
\left\{
\begin{array}{ll}
w_t - A \Delta w = 0,& x\in \rr, t>0,\\
w(x,0)=M\delta_0, &x\in \rr.\nonumber
\end{array}
\right.
\end{equation}
The uniqueness of the limit guarantees that   the sequence $\{u_\lambda\}_{\la>0}$, not only a subsequence, converges to $u_M$.

To complete the proof  it remains to analyze the case $q=2$. In order to do that, it suffices to study the convergence
of the term in \eqref{id.lim-2}. From \eqref{conv.non}  we have that for all  $\phi \in {C}^\infty_c(\mathbb{R})$
\begin{align}\label{id.lim-2.1}
\int _\tau^t \int _{\rr} u_\lambda^2 (s,x)\phi_x(x)dxds\rightarrow \int _\tau^t \int _{\rr} \ou ^2 (s,x)\phi_x(x)dxds,\mbox{ as } \lambda\rightarrow\infty.
\end{align}
Then, returning to \eqref{id.lim} the above results allow us to conclude that  $\ou\in C((0,\infty),L^1(\rr))$ is the unique solution $u_M$ of the initial value problem
\begin{equation}
\left\{
\begin{array}{ll}
w_t - A \Delta w + (\displaystyle\frac{w^2}{2})_x= 0,& x\in \rr, t>0,\\
w(x,0)=M\delta_0, &x\in \rr.\nonumber
\end{array}
\right.
\end{equation}
The uniqueness of the limit (see \cite{0762.35011})
guarantees that   the sequence $\{u_\lambda\}_{\la>0}$, not only a subsequence, converges to $u_M$.

\textbf{Step 4. Final step.}
From Steps 2 and 3 we obtain that
$$\|u_\lambda (1)-u_M(1)\|_{L^1(\rr)}\rightarrow 0,\quad \text{as}, \quad\lambda \rightarrow \infty.$$
This proves Theorem \ref{asimp} in the case $p=1$.

Classical results on the heat and Burgers equation  \cite{0762.35011} give us that the profile $u_M$ satisfies
\begin{equation*}\label{decay.profile}
\|u_M(t)\|_{L^p(\rr)}\lesssim t^{-\frac 12(1-\frac 1p)}.
\end{equation*}
Using this decay property and Theorem \ref{decay.simple} we reduce the proof of \eqref{lim.t} to the case $p=1$:
\begin{align*}
\|u(t)-u_M(t)\|_{L^p(\rr)}&\leq \|u(t)-u_M(t)\|_{L^{2p}(\rr)}^{\frac{2(p-1)}{2p-1}}\|u(t)-u_M(t)\|_{L^1(\rr)}^{1-\frac{2(p-1)}{2p-1}}\\
&\lesssim t^{-\frac 12(1-\frac 1{2p})\frac{2(p-1)}{2p-1} } \|u(t)-u_M(t)\|_{L^1(\rr)}^{1-\frac{2(p-1)}{2p-1}}\\
&\lesssim t^{-\frac 12(1-\frac 1p)}o(1).
\end{align*}
%
The proof of the main result is now finished.
\end{proof}

\begin{proof}[Proof of Proposition \ref{main}.]

{\textbf{Step I. Compactness in $L^2((0,T),H^{-1}(\Omega))$.} }
Let us consider  $\Omega$ a smooth bounded domain of $\rr$. We check the conditions in Theorem \ref{simon}.
For any $0\leq t_1<t_2<T$ we define
$$v_n(x)=\int _{t_1}^{t_2} f_n(s,x)ds.$$
Using \eqref{hyp.2} we obtain that $\{v_n\}_{n\geq 1}$ satisfies
\begin{align*}
n^2\int _{\Omega}\int _{\Omega} &\rho_n(x-y)(v_n(x)-v_n(y))^2dxdy\\
&
\leq n^2 T\int _0^T\int _{\Omega}\int _{\Omega} \rho_n(x-y)(f_n(t,x)-f_n(t,y))^2dxdydt\leq MT.
\end{align*}
Theorem \ref{rossi} gives us that there exists $v\in H^1(\Omega)$  such that, up to a subsequence,
$v_n\rightarrow v$ in $L^2(\Omega)$.
Since $L^2(\Omega)\hookrightarrow H^{-1}(\Omega)$ we obtain that $\{v_n\}_{n\geq 1}$ is  relatively compact in $H^{-1}(\Omega)$ and the first condition in Theorem \ref{simon} is satisfied. Let us now check the second condition in Theorem \ref{simon}. Using  hypothesis \eqref{hyp.3} we obtain that
$$\|\tau_hf_n-f_n\|_{L^2((0,T),H^{-1}(\Omega))}\leq \|\partial_t f_{n}\|_{L^2((0,T),H^{-1}(\Omega))}\leq
\|\partial_t f_{n}\|_{L^2((0,T-h),H^{-1}(\rr))}\leq M.$$
Theorem \ref{simon} guarantees that, up to a subsequence, $f_n\rightarrow f$ in $ L^2((0,T),H^{-1}(\Omega)). $

\medskip

\textbf{Step II. Compactness  in $L^2((0,T),L^2_{loc}(\rr))$}.
Using that $\{f_n\}_{n\geq 1}$ is uniformly bounded in $L^2((0,T)\times \rr)$ we obtain that, up to a subsequence, $f_n\rightharpoonup f$ in $L^2((0,T)\times \rr)$. Moreover, Theorem \ref{rossi-time} guarantees that $f\in L^2((0,T), H^1(\rr))$.

Let us consider  $\Omega$ a smooth bounded domain of $\rr$. We now combine the strong convergence in $L^2(0,T,H^{-1}(\Omega))$, the fact that $f\in L^2(0,T,H^1(\Omega))$ and the following Lemma  to prove the  compactness of  $\{f_n\}_{n\geq 1}$  in $L^2((0,T),L^2_{loc}(\rr))$.

\begin{lemma}\label{balance}
There exists $\delta=\delta (\rho)$ such that for every $\eps\in (0,1)$ the following
\begin{equation}\label{est.balance}
\|u\|_{L^2(\rr)}^2\leq \eps n^2 \int_{\rr}\int _{\rr} \rho_n(x-y)(u(x)-u(y))^2dxdy + \frac 2 {\eps} \|u\|_{H^{-1}(\rr)}^2
\end{equation}
holds for all $n\geq (\delta \eps)^{-1/2}$ and for all $u\in L^2(\rr)$.
\end{lemma}

We now localize   inequality \eqref{est.balance}. Let $\chi$ be a smooth function supported in $\Omega$. Thus
\begin{align*}
\|v\chi\|_{L^2(\Omega)}^2& \leq \eps n^2 \int_{\rr}\int _{\rr} \rho_n(x-y)\Big((\chi v)(x)-(\chi v)(y)\Big)^2dxdy + 2\eps^{-1}\|\chi v\|_{H^{-1}(\rr)}^2\\
&\leq  \eps n^2 \int_{\rr}\int _{\rr} \rho_n(x-y)\Big((\chi v)(x)-(\chi v)(y)\Big)^2dxdy +  {C(\chi)}\eps^{-1}\|  v\|_{H^{-1}(\Omega)}^2.
\end{align*}
We apply this inequality to $v=u-f_n $ and  integrate the new inequality  on $[0,T]$. We get
\begin{align*}
\int _0^T  \|(f-f_n)\chi\|_{L^2(\Omega)}^2dt \leq &  \eps n^2\int _0^T \int_{\rr}\int _{\rr} \rho_n(x-y)\Big((\chi (f_n-f))(x)-(\chi (f_n-f))(y)\Big)^2dxdydt\\
& + \frac {C(\chi)}\eps\int _0^T\|  f_n-f\|_{H^{-1}(\Omega)}^2dt\\
\lesssim &\eps n^2\int _0^T \int_{\rr}\int _{\rr} \rho_n(x-y)\Big((\chi f_n)(x)-(\chi f_n)(y)\Big)^2dxdydt\\
&+\eps n^2\int _0^T \int_{\rr}\int _{\rr} \rho_n(x-y)\Big((\chi f)(x)-(\chi f)(y)\Big)^2dxdydt\\
& + \frac {C(\chi)}\eps\int _0^T\|  f_n-f\|_{H^{-1}(\Omega)}^2dt.
\end{align*}

We now use the following lemma that we will prove later.

\begin{lemma}\label{local.sup}
For any $\chi\in W^{1,\infty}(\rr)$ there exists a positive constant $C(\chi)$ such that
 the following holds
\begin{align}\label{est.sup}
n^2 & \int _{\rr}\int _{\rr}  \rho_n(x-y)\Big[(\chi u)(x)-(\chi u)(y)\Big]^2dxdy\\
&
\nonumber \leq C(\chi) n^2 \int _{\rr}\int _{\rr} \rho_n(x-y)(u(x)- u(y))^2dxdy + C(\chi)\Big(\int _{\rr}\rho(z)|z|^2\Big)
\|u\|^2_{L^2(\rr)}
\end{align}
for any positive integer $n$ and for any $u\in L^2(\rr)$.
\end{lemma}

Applying Lemma \ref{est.ariba}, Lemma \ref{local.sup} and assumption \ref{hyp.2}  we get
\begin{align*}
\int _0^T  \|(f-f_n)&\chi\|_{L^2(\Omega)}^2dt \\
\leq &\eps  C(\chi) \Big[ n^2\int _0^T \int_{\rr}\int _{\rr} \rho _n(x-y)\Big( f_n(t,x)- f_n(t,y)\Big)^2dxdydt+
\int _0^T \|f_n(s)\|_{L^2(\rr)}^2ds\Big]\\
&+\eps \|\chi f\|^2_{L^2(0,T,H^1(\rr))} + \frac {C(\chi)}\eps\int _0^T\|  f_n-f\|_{H^{-1}(\Omega)}^2dt\\
\leq &\eps C(\chi)M +\frac {C(\chi)}\eps \|f_n-f\|^2_{L^2((0,T),H^{-1}(\Omega))}+\eps  \|f\|_{L^2((0,T),H^1(\rr))}^2.
\end{align*}

Let us now take any  $\eps'>0$. We choose $\eps=\min \{{\eps'}/4M, \eps' /4 \|f\|_{L^2((0,T),H^1(\rr))}^2 \}$ and $n\geq N_0=(\delta \eps)^{1/2}$. Choosing
large $N_1\geq N_0$ we have that, up to a subsequence,
$$ \|f_n-f\|^2_{L^2((0,T),H^{-1}(\Omega))}\leq \frac{\eps \eps'}{2C(\chi)}, \quad  n\geq N_1.$$
Hence, up to a subsequence,
\[
\|(f-f_n)\chi \|^2_{L^2((0,T),L^2(\Omega))}\leq {\eps'},\quad n\geq N_1.
\]
This shows that $\{f_n\}_{n\geq 1}$ is relatively compact  in $L^2((0,T), L^2_{loc}(\Omega))$. Since $\Omega$ has been chosen arbitrarily, a diagonal argument
guarantees that $\{f_n\}_{n\geq 1}$ is relatively compact in $L^2((0,T), L^2_{loc}(\rr))$. Moreover, there exists $f\in L^2((0,T), H^1(\rr))$ such that, up to a subsequence
$f_n\rightarrow f$ in $L^2((0,T), L^2_{loc}(\rr))$.

This finishes the proof of the proposition.
\end{proof}

%
%
%
%
%
%
%
%
%

\begin{proof}[Proof of Lemma \ref{balance}]
Let us denote $\eta(\eps)=2/\eps$.
Using the Fourier transform we have to prove that
\[
1\leq \eps n^2 \Big(1-\hat \rho(\frac \xi n)\Big)+\frac{\eta (\eps)}{1+\xi^2}, \quad \forall \, \xi\in \rr.
\]
Since $\rho$ is a smooth function, there exist $R$ and $\delta>0$ such that
\[
\left\{
\begin{array}{ll}
\hat \rho(\xi)\leq 1-\frac {\xi^2}2,& |\xi |\leq R,\\[10pt]
\hat \rho(\xi)\leq 1-\delta, & |\xi|\geq R.
\end{array}
\right.
\]

Hence for $|\xi/n|\geq R$ we have to check that
$$1\leq \eps n^2 \delta +\frac{\eta (\eps)}{1+|\xi|^2}.$$
Choosing $n\geq (\eps \delta )^{-1/2}$ the last inequality holds.
For $|\xi/n|\leq R$ we have to check that
$$1\leq \frac{\eps |\xi|^2}{2}+\frac{\eta (\eps)}{1+|\xi|^2}$$
or equivalently
$$\eps \xi ^4 +\xi ^2(\eps-2)+2(\eta(\eps)-1)\geq 0.$$
Observe that $\eta(\eps)=2/\eps\geq (\eps+2)^2/8\eps$ and the last inequality is satisfied for all $\eps \in (0,1)$ and $\xi\in \rr$.
\end{proof}

%
%

\begin{proof}[Proof of Lemma \ref{local.sup}]
Using the identity
\[
(\chi u)(x)-(\chi u)(y)=\chi(x)(u(x)-u(y))+u(y)(\chi(x)-\chi(y))
\]
we obtain that
\begin{align*}
n^2  \int _{\rr}\int _{\rr} &  \rho_n(x-y)\Big((\chi u)(x)-(\chi u)(y)\Big)^2dxdy\\
&\lesssim
n^2 \|\chi\|_{L^\infty(\rr)}^2  \int _{\rr}\int _{\rr}  \rho_n(x-y)( u(x)- u(y))^2dxdy\\
&\quad +n^2 \int _{\rr}\int _{\rr}  \rho_n(x-y) u^2(y)(\chi(x)-\chi(y))^2dxdy.
\end{align*}
The second term in the right hand side of the above inequality satisfies:
\begin{align*}
n^2 \int _{\rr}\int _{\rr}&  \rho_n(x-y) u^2(y)\Big(\chi(x)-\chi(y)\Big)^2dxdy\\
&\leq n^2 \int _{\rr}\int _{\rr} \rho_n(x-y) u^2(y)(x-y)^2 \int _0^1 (\chi ')^2(y+s(x-y)) dsdxdy\\
&=\int _{\rr} \rho(z)|z|^2 \int _{\rr} u^2(y)  \int _0^1 (\chi ')^2(y+sz) dsdydz\\
&\leq  \|\chi\|_{W^{1,\infty(\rr)}} ^2\int _{\rr} \rho(z)|z|^2 dz\int _{\rr} u^2(y)dy.
\end{align*}
The proof is now complete.
\end{proof}

 {\bf
Acknowledgements.}
This work started  when both authors visited ICMAT, Madrid, Spain in the spring of 2012. The authors thank the center and Rafael Orive for hospitality and support.
The authors also thank  Juan Luis Vazquez for fruitful discussions.

The first author was partially supported by Grant PN-II-ID-PCE-2011-3-0075 of the Romanian National Authority for Scientific Research, CNCS -- UEFISCDI, MTM2011-29306-C02-00, MICINN, Spain and ERC Advanced Grant FP7-246775 NUMERIWAVES. The second author was partially supported by
CNPq and PRONEX (Brazil).

%
%
%

%
%
\bibliographystyle{plain}
\bibliography{biblio}

\end{document}